\documentclass[12pt, a4paper]{article}

\usepackage{a4,amsmath,amssymb, amsthm, latexsym, color, graphicx,url}
\usepackage{tikz}
\usetikzlibrary{arrows.meta}
\usetikzlibrary{decorations.markings}
\usetikzlibrary{calc}

\usepackage{authblk}
\usepackage{fullpage}
\usepackage{enumitem}
\usepackage{xcolor}
\newtheorem{theorem}{Theorem}[section]

\newtheorem{lemma}[theorem]{Lemma}
\newtheorem{corollary}[theorem]{Corollary}

\theoremstyle{definition}
\newtheorem{example}[theorem]{Example}
\newtheorem{definition}[theorem]{Definition}
\newtheorem{remark}[theorem]{Remark}

\newcommand{\E}{\overrightarrow{E}}

\title{Digraph-defined external difference families and new circular external difference families}
\author[1]{Sophie Huczynska}
\author[2]{Christopher Jefferson}
\author[3]{Struan McCartney}
\affil[1]{School of Mathematics and Statistics, University of St Andrews, St Andrews, KY16 9SS, Scotland, UK; email: sh70@st-andrews.ac.uk}
\affil[2]{School of Science and Engineering, University of Dundee, Dundee, DD1 4HN, Scotland, UK; email: cjefferson001@dundee.ac.uk}
\affil[3]{School of Mathematics and Statistics, University of St Andrews, St Andrews, KY16 9SS, Scotland, UK; email: sm444@st-andrews.ac.uk}
\date{MSC code: 05B10}
\begin{document}
\maketitle

\begin{abstract}
External difference families (EDFs) are combinatorial objects which were introduced in the early 2000s, motivated by information security applications such as the construction of AMD codes.  Various generalizations have since been defined and investigated, in particular strong external difference families (SEDFs) and circular external difference families (CEDFs).  In this paper, we present a framework based on graphs and digraphs which offers a new unified way to view these structures, and leads to natural new research questions. We present constructions and structural results about these digraph-defined EDFs, and we obtain new explicit constructions for infinite families of CEDFs, in particular $(ml^2+1,m,l,1)$-CEDFs.  Our techniques include cyclotomy in finite fields and direct constructions in cyclic groups and direct products of cyclic groups.  We construct the first infinite family of such CEDFs in non-cyclic abelian groups; these have odd values of $m$ and $l$. We also present the first CEDF in a non-abelian group. 
\end{abstract}

\begin{section}{Introduction}
External difference families (EDFs) are combinatorial objects which were introduced in the early 2000s, motivated by information security applications \cite{Cramer, Oga}.  Several variants of these have been introduced subsequently - e.g. strong external difference families (SEDFs) \cite{PatSti1}, circular external difference families (CEDFs) \cite{VeiSti} and others. Such a family consists of a collection $\mathcal{A}$ of disjoint same-size subsets of a group $G$, with the property that the multiset of pairwise differences between elements of certain distinct sets in $\mathcal{A}$ contains every non-identity element of $G$ the same number of times.  This is an external analogue of traditional difference families, which have been studied since the 1930s; for these, the pairwise differences between elements within each set are considered, and their multiset union is required to contain each non-identity element of $G$ equally often. Relaxations of the conditions (e.g. dropping the requirement of equal set-sizes, leading to generalised EDFs (GEDFs) and generalised SEDFs (GSEDFs)) have also been explored \cite{PatSti1}.

Connections between graphs and certain internal/external difference families are well known, and have been explored for example in \cite{BurGio} and \cite{HucPat}. In these papers, the vertices of the graph correspond to the elements of the group $G$ and (labelled) edges between them represent the differences. This provides a link to graph decomposition problems.  In \cite{Bur}, the concept of strong difference family (SDF) is introduced as a collection of multisets such that the multiset union of their internal differences contains every element of $G$ (including the identity) equally often. In \cite{BurGio}, a generalised definition of SDF is introduced, in which the set of internal differences is defined via the edges of a digraph $\Gamma$, and the original definition is retrieved upon taking $\Gamma$ to be complete. In \cite{HucPat}, graph decompositions are considered which correspond directly to certain EDFs in $\mathbb{Z}_n$ via the process of development.

In this paper, we present a different connection between graph theory and external difference families and their generalizations, inspired by - but distinct from - these ideas.  We observe that the sets of such a family may be associated with the vertices of a digraph, with directed edges between precisely those pairs of vertices for which the external (directed) differences between the corresponding sets contribute to the multiset of external differences.  This provides a natural and useful framework in which to view EDFs and their variants, particularly SEDFs and CEDFs.

We note that a somewhat-related idea arises implicitly in \cite{PatSti2} when an initial directed cycle is constructed, then a blow-up construction is applied to it to create a CEDF in a cyclic group.  However our definition does not require any construction-based relationship between the digraph and the sets involved.  The ``graceful directed graphs" of \cite{BloHsu} give examples of our structures in the specific case when the group is cyclic and every set is a singleton set.

Our work was underpinned by search for examples in GAP \cite{GAP}, using a constraint-satisfaction modelling language \cite{Akg} and solver \cite{GenJefMig}.

The paper is structured as follows: we first set up the necessary background and new definitions.  We present constructions and examples of EDFs defined by complete graphs, cycles and complete bipartite graphs (both oriented and undirected). Those defined by oriented cycles (with the natural orientation) correspond to CEDFs, and we present various new results on CEDFs.  In particular, we present a new infinite family of $(ml^2+1,m,l,1)$-CEDFs in non-cyclic abelian groups; these have both $m$ and $l$ odd.  We also give the first example of a CEDF in a non-abelian group.  We end by indicating further research questions emerging from this work.
\end{section}

\begin{section}{Background}\label{sec:background}
Throughout, $G$ will denote a finite group.  Unless otherwise stated, we will write $G$ additively. For subsets $A, B$ of $G$, we define the multiset $\Delta(A,B)=\{a-b:a\in A, b \in B\}$.  All unions are multiset unions unless otherwise stated.

We will frequently work in $\mathbb{Z}_n$, the additive group of integers modulo $n$. We consider its elements as $\{0,1,\ldots, n-1\}$, with the natural order $0<1< \cdots < n-1$.  For $a, b \in \mathbb{Z}_n$, we refer to the set of consecutive elements $\{a,a+1,\ldots,b\}$ as the \emph{interval} $[a,b]$.  

The concept of classical external difference family was first defined in \cite{Oga}, motivated by an application to AMD codes:

\begin{definition}\label{def:EDF}
Let $G$ be a group of order $n$ and let $m>1$.  A family of disjoint $l$-sets $\{A_1, \ldots, A_m\}$ in $G$ is an $(n,m,l,\lambda)$-EDF if the multiset equation
$ \bigcup_{\{i,j: i \neq j\}} \Delta(A_i,A_j)= \lambda(G \setminus \{0\})$ holds.
\end{definition}

The following stronger version of an EDF was first defined in \cite{PatSti1}, corresponding to a stronger security model. (Note this is distinct from the notion of an SDF mentioned previously.)
\begin{definition}\label{def:SEDF}
Let $G$ be a group of order $n$ and let $m>1$.  A family of disjoint $l$-sets $\{A_1, \ldots, A_m\}$ in $G$ is an $(n,m,l,\lambda)$-SEDF if, for each $i$ with $1 \leq i \leq m$, the multiset equation
$ \bigcup_{\{j: j \neq i\}} \Delta(A_i,A_j)= \lambda(G \setminus \{0\})$ holds.
\end{definition} 
Note that a strong EDF will always be an EDF, but the converse need not be true.

Recently, Stinson and Veitch introduced new objects called circular external difference families \cite{VeiSti}. 

\begin{definition}
Let $G$ be a group of order $n$. Suppose $m>1$ and $1 \leq c \leq m-1$.   A family of disjoint $l$-sets $\{A_1, \ldots, A_m\}$ in $G$ is an $(n, m, l,\lambda)$-$c$-CEDF if the following multiset equation holds:
$ \bigcup_{i=0}^{m-1} \Delta(A_{i+c \mod m}, A_i)=\lambda (G \setminus \{0\})$.
\end{definition}
If $c=1$ then the $c$ is sometimes omitted from the notation (as in \cite{PatSti2}).  We note that when $m>3$, CEDFs are not special cases of the standard EDFs; a $(n,3,l,\lambda)$-CEDF is a $(n,3,l,2\lambda)$-EDF.

We present the following observation which clarifies the structure of $c$-CEDFs:
\begin{lemma}\label{lemma:CEDF}
Let $\mathcal{A}=\{A_1, \ldots, A_m\}$ be a disjoint collection of $l$-subsets of a group $G$. Suppose $\mathcal{A}$ is a $c$-CEDF.  Then, if $\gcd(c,m)=d$, the multiset in the definition may be written as a disjoint union of $d$ multiset unions, each involving $m/d$ sets, as follows (where all indices are taken modulo $m$):
$$ \bigcup_{i=0}^{m-1} \Delta(A_{i+c}, A_i)=\bigcup_{j=0}^{d-1} \left(\bigcup_{i=0}^{(m/d-1)} \Delta(A_{(i+1)c+j},A_{ic+j}) \right)=\lambda(G \setminus \{0\})$$
\end{lemma}
\begin{proof}
The decomposition of the $m$ sets into $d$ disjoint collections of $m/d$ sets follows from the result on permutations that, if $f$ is a cycle of length $m$, $f^c$ decomposes as the product of $\gcd(c,m)$ disjoint cycles of length $m/\gcd(c,m)$.  The rest of the result follows from the definition.
\end{proof}

The following stronger version of a CEDF has been defined \cite{VeiSti}.

\begin{definition}
Let $G$ be a group of order $n$. Suppose $m \geq 2$ and $1 \leq c \leq m-1$.   A family of disjoint $l$-sets $\{A_1, \ldots, A_m\}$ in $G$ is an $(n, m, l,\lambda)$-$c$-SCEDF if, for each $0 \leq i \leq m-1$, the multiset equation $\Delta(A_{i+c}, A_i)=\lambda (G \setminus \{0\})$ holds (indices taken modulo $m$).
\end{definition}

It is observed in \cite{WuYanFen} that, due to the decomposition structure, any $c$-SCEDF can be viewed as the disjoint union of $1$-SCEDFs; the authors of \cite{WuYanFen} define a $c$-SCEDF to be \emph{trivial} if it is the disjoint union of 2-set $1$-SCEDFs (ie 2-set SEDFs). Using character theory it is proved that any $1$-SCEDF in an abelian group must have $m=2$ sets, and hence that any SCEDF must be trivial.  Note that the use of the word \emph{trivial} in this context does not imply that all subsets have size $1$ (in contrast to the trivial difference set which is taken in the literature to be any singleton set).  For example, in $\mathbb{Z}_{17}$ a trivial $2$-SCEDF is given by sets $A_0=\{1,13,16, 4\}, A_1=\{3,5,14,12\}, A_2=\{9,15,8,2\}$ and $A_3=\{10,11,7,6\}$.
\end{section}

\begin{section}{A new viewpoint}
We now introduce a way of defining external difference families and their generalizations in terms of graphs and digraphs.

Throughout this paper, a graph or digraph $G$ on $m$ vertices will have vertex set $V(G)=\{0,\ldots,m-1\}$.  All graphs and digraphs will be finite and simple (no loops or multiple edges), and they will be labelled.

For a digraph $G$, we denote by $\E(G)$ its set of directed edges.   A graph is said to be \emph{oriented} if at most one of $(i,j)$ or $(j,i)$ is in the directed edge-set for each pair $i \neq j$.  

For an undirected graph $G$ (which will simply be referred to as a graph) we will view it as a digraph, by replacing each undirected edge by a pair of inverse directed edges.   We define the following notation.  Let the edge-set of $G$ as an undirected graph be $E(G)=\{ \{i,j\}: i,j \in V(G) \}$; then we define the \emph{directed edge-set} $\E(G)$ of $G$ to be 
$$ \E(G):=\{ (i,j) \in V(G) \times V(G): \{i,j\} \in E(G)\}.$$

\begin{example}
Consider $K_3$ with $V(K_3)=\{0,1,2\}$ and $E(K_3)=\{ \{0,1\}, \{1,2\}, \{2,0\} \}$; then $\E(K_3)=\{(0,1), (0,2), (1,0), (1,2), (2,0), (2,1)\}$.  
\end{example}

All the digraphs which will be considered in this paper will either be oriented digraphs or undirected graphs viewed as digraphs as above.

We are now ready to define our main object of study.

\begin{definition}\label{def:H-defined EDF}
Let $G$ be a group of order $n$, and let $\mathcal{A}=(A_0,\ldots,A_{m-1})$ be an ordered collection of disjoint subsets of $G$, each of size $l$.  Let $H$ be a labelled digraph on $m$ vertices $\{0,1,\ldots,m-1\}$ and let $\E(H)$ be the set of directed edges of $H$.  Then $\mathcal{A}$ is said to be an $(n,m,l,\lambda; H)$-EDF if the following multiset equation holds:
$$ \bigcup_{(i,j) \in \E(H)} \Delta(A_j,A_i) = \lambda (G \setminus \{0\}).$$
We will call such a structure a \emph{digraph-defined} EDF.  If we wish to emphasise $H$, we will call it an $H$-defined EDF.
\end{definition}

We introduce the following notation and labelling for commonly-used graphs and digraphs which will appear in this paper. To distinguish between an undirected and oriented version of a given underlying graph, we denote the oriented version by a superscript $*$.
\begin{itemize}
\item  Complete graph $K_m$:  $V(K_m)=\{0,1,\ldots, m-1\}$ and  $\E(K_m):=\{ (i,j): 0 \leq i,j\leq m-1, i \neq j \}$. 
\item Oriented complete digraph $K_m^*$ (also known as tournament): $V(K_m^*)=\{0,1,\ldots, m-1\}$; the standard set of directed edges will be  $\E(K_m^*):=\{ (i,j): 0 \leq i <j\leq m-1 \}$ (but we will also consider other orientations).
\item Cycle graph $C_m$:  $V(C_m)=\{0,1,\ldots, m-1\}$ and  $\E(C_m):=\{ (i,j) : j \equiv i+1 \mod m \mbox{ or } i \equiv j+1 \mod m: 0 \leq i, j \leq m-1 \}$. 
\item Oriented cycle $C_m^*$:  $V(C_m^*)=\{0,1,\ldots, m-1\}$; the standard set of directed edges will be  $\E(C_m^*):=\{ (i,i+1 \mod m): 0 \leq i \leq m-1 \}$. 
\item Complete bipartite graph $K_{a,b}$: bipartition $V(K_{a,b})=A \cup B$ where $A=\{0,\ldots,a-1\}$ and $B=\{a,\ldots, a+b-1\}$ and $\E(K_{a,b}):=\{ (i,j): i \in A, j \in B \mbox{ or } i \in B, j \in A\}$. 
\item Oriented complete bipartite digraph $K_{a,b}^*$: bipartition $A \cup B$ where $A=\{0,\ldots,a-1\}$ and $B=\{a,\ldots, a+b-1\}$; the standard set of directed edges will be $\E(K_{a,b}^*):=\{ (i,j): i \in A, j \in B \}$.
\end{itemize}

If $H$ with $V(H)=\{0,\ldots,m-1\}$ is a disjoint union of graphs $H_1 \cup \cdots \cup H_u$, with $|V(H_i)|=h_i$, we will take $V(H)=V(H_1) \cup \cdots \cup V(H_u)$ where $V(H_1)=\{0,\ldots, h_1-1\}$, $V(H_2)= \{h_1,\ldots, h_1+h_2-1\}$, \ldots, $V(H_u)=\{h_1+\cdots + h_{m-2}, \ldots, m -1\}$.

For an $H$-defined digraph when $H$ is (an undirected or oriented) $K_{a,b}$ with bipartition $A \cup B$, we will use a semi-colon to separate the sets corresponding to the vertices of $A$ from those corresponding to the vertices of $B$.  If $H$ is a disjoint union of $H_1 \cup \cdots \cup H_u$, we will similarly use semicolons to separate the sets corresponding to the vertices of distinct $H_i$.

\begin{example}\label{ex:Z13}
In $G=(\mathrm{GF}(13),+)$, let $A_0=\{1,5,8,12\}, A_1=\{2,3,10,11\}$ and $A_2=\{4,6,7,9\}$.  Then $(A_0,A_1,A_2)$ is a $(13,3,4,8; C_3)$-EDF, a $(13,3,4,4; C_3^*)$-EDF, a $(13,3,4,8; K_3)$-EDF and a $(13,3,4,4; K_3^*)$-EDF.
\end{example}

\begin{remark}
The structures presented in Section \ref{sec:background} can be put into this new framework as follows.
Let $\mathcal{A}=\{A_0,\ldots, A_{m-1}\}$ be a collection of disjoint $l$-subsets of a group $G$.
\begin{itemize}\label{main:EDFs}
\item $\mathcal{A}$ is an $(n,m,l,\lambda)$-EDF precisely if $(A_0,\ldots,A_{m-1})$ is an $(n,m,l,\lambda;K_m)$-EDF.
\item $\mathcal{A}$ is an $(n,m,l,\lambda)$-SEDF precisely if, for each $i$, $(A_0,A_1,\ldots, A_{i-1},A_{i+1},\ldots A_{m-1}; A_i)$ is an $(n,m,l,\lambda; K_{m-1,1}^*)$-EDF.
\item $\mathcal{A}$ is an $(n,m,l,\lambda)$-$1$-CEDF precisely if $(A_0,\ldots,A_{m-1})$ is an $(n,m,l,\lambda; C_m^*)$-EDF.
\item If $c>1$ and $\gcd(c,m)=1$, $\mathcal{A}$ is an $(n,m,l,\lambda)$-$c$-CEDF precisely if $(A_0, A_c, A_{2c},\ldots,A_{c(m-1)})$ is an $(n,m,l,\lambda; C_m^*)$-EDF (where indices are taken modulo $m$).
\item If $\gcd(c,m)=d>1$, $\mathcal{A}$ is an $(n,m,l,\lambda)$-$c$-CEDF precisely if 
$$(A_0,A_{c}, \ldots, A_{(m/d-1)c}; A_1, A_{1+c}, \ldots, A_{1+(m/d-1)c}; \ldots; A_{(d-1)}, \ldots, A_{(d-1)+(m/d-1)c})$$ 
is an $(n,m,l,\lambda; H)$-EDF where $H$ is the disjoint union of $d$ oriented cycles $C_{m/d}^*$ if $m/d \geq 3$ and the disjoint union of $d$ undirected paths $P_1$ if $m/d=2$.
\end{itemize}
\end{remark}

\begin{example}
The $(17,4,4,4)$-$2$-CEDF in $\mathbb{Z}_{17}$ given by sets $A_0=\{1,13,16, 4\}, A_1=\{3,5,14,12\}, A_2=\{9,15,8,2\}$ and $A_3=\{10,11,7,6\}$ is a $(17,4,4,4; H)$-EDF, where $V(H)=\{0,1,2,3\}$ and $H$ is the disjoint union $H_1 \cup H_2$ where $\E(H_1)=\{(0,2), (2,0)\}$ and $\E(H_2)=\{(1,3), (3,1)\}$.
\end{example}
\end{section}

Finally we present relationships between EDFs defined by directed and undirected versions of the same underlying graph.

\begin{theorem}\label{general:dir_to_undir}
Let $G$ be a group of order $n$.  Let $H$ be a graph on $m$ vertices and let $H^*$ denote any orientation of $H$.
If $\mathcal{A}$ is an $(n,m,l,\lambda;H^*)$-EDF, then $\mathcal{A}$ is an $(n,m,l,2\lambda;H)$-EDF.
\end{theorem}

We also have the following partial converse.

\begin{theorem}\label{general:ij=ji}
Let $G$ be a group of order $n$.  Let $H$ be a graph on $m$ vertices and let $H^*$ denote any orientation of $H$.
If $\mathcal{A}=(A_0,\ldots, A_{m-1})$ is an $(n,m,l,\lambda;H)$-EDF and $\Delta(A_i,A_j)=\Delta(A_j,A_i)$ for all edges $\{i,j\}$ of $H$, then $\lambda$ is even and $\mathcal{A}$ is an $(n,m,l,\lambda/2;H^*)$-EDF.
\end{theorem}
\begin{proof} 
Denote by $\overline{H^*}$ a copy of $H$ with the reverse orientation to that of $H^*$. So $(i,j) \in \E(\overline{H^*})$ precisely if $(j,i) \in \E(H^*)$ and we have ${\E}(H)=\overrightarrow{E}(H^*)\cup\overrightarrow{E}(\overline{H^*})$.  Each edge $\{i,j\} \in E(H)$ is the union of $(i,j) \in \overrightarrow{E}(H^*)$ and $(j,i) \in \overrightarrow{E}(\overline{H^*})$. For each $\{i,j\} \in E(H)$, since $\Delta(A_i,A_j)=\Delta(A_j,A_i)$, there is a contribution of $\Delta(A_i,A_j)$ to the difference multisets of the $H^*$-defined EDF and of the $\overline{H^*}$-defined EDF, corresponding to a contribution of $2 \Delta(A_i,A_j)$ to the difference multiset of the EDF, which equals $\lambda(G \setminus \{0\})$. The result follows.
\end{proof}

While these relationships are helpful in understanding the links between directed and undirected versions, they are far from describing the whole picture.  There are many examples of both type of EDF which do not result from applying Theorems \ref{general:dir_to_undir} or \ref{general:ij=ji} to an EDF of the other type.

\begin{section}{$H$-defined EDFs when $H$ is complete}

As we have seen, the $H$-defined EDFs when $H=K_m$ precisely correspond to the standard EDFs.  These structures have been much-studied and so we will mention them only briefly here.  In contrast, those with $H=K_m^*$ do not - to our knowledge - correspond to previously-investigated types of external difference families.

\begin{subsection}{External difference families}
The standard EDFs have received considerable attention and many results about them are known \cite{HucPat, PatSti1}. However, it is perhaps worth noting that, due to their relative lack of structure compared to other variations of the definition, there are fewer known general constructions for EDFs than there are for EDFs with extra conditions such as SEDFs.

In order to present a useful general construction, we require a few facts about cyclotomy.  For more information on cyclotomy, see \cite{Sto}.  Let $q$ be a power of a prime $p$ and let $ {\rm GF}(q)$ denote the finite field of order $q$.  Let $\alpha$ be a primitive element of $ {\rm GF}(q)$.  

\begin{definition}\label{def:cclass}
Let $q=ef+1$ where $e,f \in \mathbb{N}$.  The cyclotomic classes $C_i^e$ in $ {\rm GF}(q)$ of order $e$ ($0 \leq i \leq e-1$) are defined as:
$$ C_i^e=\{\alpha^{es+i}: 0 \leq s \leq f-1\}.$$
Here, $C_0^e$ is the multiplicative subgroup of $ {\rm GF}(q)$ of cardinality $f$.
\end{definition}

We now present the following well-known EDF construction \cite{ChaDin, DavHucMul}:

\begin{theorem}\label{thm:cycEDF}
Let $q=ef+1$ be a prime power.  Then the set $\{C_0^e, \ldots, C_{e-1}^e\}$ of all cyclotomic classes of order $e$ in $GF(q)$ forms a $(q,e,f,(e-1)f;K_e)$-EDF. 
\end{theorem}

\begin{example}\label{ex:K_3EDF}
In $GF(13)$, taking $e=3$ and $\alpha=2$, we have that $(C_0^3,C_1^3,C_2^3)$  forms a $(13,3,4,8;K_3)$-EDF, where $C_0^3=\{1,5,12,8\}, C_1^3=\{2,10,11,3\}$ and $C_2^3=\{4,7,9,6\}$.
\end{example}

\end{subsection}

\begin{subsection}{Tournament-defined EDFs}

We establish the following results for EDFs defined by tournaments. By Theorem \ref{general:dir_to_undir}, any EDF defined by a tournament yields an EDF in the standard sense, while by Theorem \ref{general:ij=ji}, a standard EDF which satisfies an extra condition may be used to obtain a tournament-defined EDF.

We present a direct cyclotomic construction for tournament-defined EDFs.  We first need a technical lemma (Lemma 3.13 from \cite{HucJoh}):

\begin{lemma}\label{lem:SophieLaura}
Let $q=ef+1$ be a prime power.  If either $e$ is odd, or $e$ is even and $q \equiv 1 \mod 2e$, then $-1 \in C_0^e$.
\end{lemma}

\begin{theorem}\label{thm:cyc_tournament}
Let $q=ef+1$ be a prime power such that either $e$ is odd, or $e$ is even and $q \equiv 1 \mod 2e$.
Let $K_e^*$ be a tournament with an arbitrary orientation.  Let $\mathcal{A}=(C_0^e, \ldots, C_{e-1}^e)$.
Then $\mathcal{A}$ forms a $(q,e,f,(e-1)f/2; K_e^*)$-EDF in $GF(q)$.
\end{theorem}
\begin{proof}
By Theorem \ref{thm:cycEDF}, $\mathcal{A}$ is a $(q,e,f,(e-1)f); K_e)$-EDF.  In both cases (i) and (ii) of Lemma \ref{lem:SophieLaura},  $-1 \in C_0^e$, hence $C_i^e=-C_i^e$ for all $0 \leq i \leq e-1$ and so $\Delta(C_i^e,C_j^e)=\Delta(C_j^e,C_i^e)$ for all $0 \leq i,j \leq m-1$. In both cases, $(e-1)f$ is even and the result follows by  Theorem \ref{general:ij=ji}.
\end{proof}

\begin{example}
In $\mathrm{GF}(13)$, consider the cyclotomic classes $C_0^3,C_1^3,C_2^3$ of order $3$. By Theorem \ref{thm:cyc_tournament}, since $e=3$ is odd, $(C_0^3,C_1^3,C_2^3)$ form a $(13,3,4,4; K_3^*)$-EDF, where $K_3^*$ is a tournament on $3$ vertices with arbitrary orientation.  Note that these same sets were discussed in Example \ref{ex:Z13} and Example \ref{ex:K_3EDF}.   
\end{example}

\end{subsection}

\end{section}

\begin{section}{$H$-defined EDFs when $H$ is a cycle or union of cycles}

We have seen that CEDFs correspond to oriented cycles, or disjoint unions of same-size oriented cycles, with the standard orientation (where we include a single undirected edge as a directed cycle of length $2$).  

\begin{subsection}{Circular external difference families}
CEDFs were introduced in \cite{VeiSti} and various results have been established about them.  In \cite{VeiSti}, sufficient conditions for the existence of certain cyclotomic CEDFs in the additive group of a finite field were given in terms of primitive elements of the field.  The cyclotomic approach was further extended in \cite{WuYanFen}; this paper also contains structural characterizations and non-existence results for strong CEDFs. In \cite{PatSti2}, existence of infinite families of CEDFs with $\lambda=1$ in cyclic groups was established using graceful labellings of graphs. (We note that the use of graceful labellings in the strong difference family context appeared in \cite{BurGio}).  In \cite{BurMerTra}, constructions are given in cyclic groups for CEDFs with $\lambda=1$, for cases not previously covered.  In this section, we present various new results for CEDFs.

In \cite{PatSti2}, Theorem 1.8 states three main results on $(ml^2+1,m,l,1)$-$1$-CEDFs in abelian groups.  The second part asserts that if $l$ and $m$ are odd then there is no $(ml^2 + 1, m, l, 1)$-$1$-CEDF.  It seems that this result holds only in the context of cyclic groups, since we show that an infinite family of $1$-CEDFs with such parameters exist in non-cyclic abelian groups.

\begin{subsection}{CEDFs in abelian non-cyclic groups and non-abelian groups}

We present a construction for an infinite family of CEDFs in abelian groups which are not cyclic nor elementary abelian.

\begin{theorem}\label{thm:noncyclic}
Let $l \equiv 3 \mod 4$ ($l \in \mathbb{N}$). Denote $z=\frac{3}{4}(l-1)^2 \in \mathbb{N}$. Define the following subsets of $\mathbb{Z}_{\frac{3l^2+1}{2}}\times\mathbb{Z}_2$:
\begin{itemize}
\item $A_0=\bigcup_{i=0}^{l-1}\{(i,0)\}$
\item $A_1=\bigcup_{i=0}^{l-1}\{(z(i-1)-l-i,i+1)\}$
\item $A_2=\bigcup_{i=0}^{l-1}\{(zi-l,i) \}$
\end{itemize}
where the first component is taken modulo $\frac{3l^2+1}{2}$ and the second component is taken modulo $2$.\\
Then $(A_0,A_1,A_2)$ form a $(3l^2+1,3,l,1)$-CEDF in the non-cyclic abelian group $\mathbb{Z}_{\frac{3l^2+1}{2}}\times\mathbb{Z}_2$.
\end{theorem}

To aid the reader's intuition, we first present some specific examples of Theorem \ref{thm:noncyclic}. While the statement of the construction may look uninuitive, the occurrence of the differences follows a very natural and regular pattern, as can be seen from the associated subtraction tables.

\begin{example}
We consider the following special cases of Theorem \ref{thm:noncyclic}.
\begin{itemize}
\item[(i)] Take $l=3$ in Theorem \ref{thm:noncyclic}.  We obtain a $(28,3,3,1)$-CEDF in $\mathbb{Z}_{14} \times \mathbb{Z}_2$ with sets $A_0=\{(0,0), (1,0), (2,0)\}$, $A_1=\{(8,1), (10,0), (12,1)\}$ and $A_2=\{(11,0), (0,1), (3,0)\}$.  The subtraction table is shown in Table \ref{table_noncyc_28}.
\item[(ii)] Take $l=7$ in Theorem \ref{thm:noncyclic}.  We obtain a $(148,3,7,1)$-CEDF in $\mathbb{Z}_{74} \times \mathbb{Z}_2$ with sets 
\begin{itemize}
\item $A_0=\{(0,0), (1,0), (2,0), (3,0),  (4,0), (5,0), (6,0)\}$
\item $A_1=\{(40,1),(66,0),	(18,1), (44,0), (70,1), (22,0), (48,1)\}$ 
\item $A_2=\{(67,0), (20,1), (47,0), (0,1), (27,0), (54,1),	(7,0) \}$.
\end{itemize}
The subtraction tables for $\Delta(A_1,A_0), \Delta(A_2,A_1)$ and $\Delta(A_0,A_2)$ are given in Table \ref{table:A_1A_0}, Table \ref{table:A_2A_1} and Table \ref{table:A_0A_2} respectively. 
\end{itemize}

\begin{table}[h]
\renewcommand{\arraystretch}{1.2}
\centering
\setlength\tabcolsep{3pt}
\begin{tabular}{|c|ccc|ccc|ccc|} 
\hline
$-$ & $(0,0)$ & $(1,0)$ & $(2,0)$ & $(8,1)$ & $(10,0)$ & $(12,1)$ & $(11,0)$ & $(0,1)$ & $(3,0)$\\ \hline
$(0,0)$ &  &  &  &  & &  & $(3,0)$ & $\textbf{(0,1)}$ & $(11,0)$	\\
$(1,0)$  &  &  &  &  & &  & $(4,0)$ & $\textbf{(1,1)}$ & $(12,0)$	\\
$(2,0)$  &  &  &  &  & &  & $(5,0)$ & $\textbf{(2,1)}$ & $(13,0)$	\\
\hline
$(8,1)$ & $\textbf{(8,1)}$  & $\textbf{(7,1)}$ & $\textbf{(6,1)}$ &  & &  &  & & 	\\
$(10,0)$ & $(10,0)$  & $(9,0)$ & $(8,0)$ &  & &  &  & & 	\\
$(12,1)$ & $\textbf{(12,1)}$  & $\textbf{(11,1)}$ & $\textbf{(10,1)}$ &  & &  &  & & 	\\
\hline
$(11,0)$ &  &  &  & $\textbf{(3,1)}$ & $(1,0)$ & $\textbf{(13,1)}$	& & & \\
$(0,1)$  &  &  &  & $(6,0)$ & $\textbf{(4,1)}$ & $(2,0)$	& & & \\
$(3,0)$  &  &  &  & $\textbf{(9,1)}$ & $(7,0)$ & $\textbf{(5,1)}$	& & & \\
\hline
\end{tabular}
	
\caption{Subtraction table for the construction in $\mathbb{Z}_{14} \times \mathbb{Z}_2$ from Theorem \ref{thm:noncyclic}}
\label{table_noncyc_28}
\end{table}

\begin{table}[h]
\renewcommand{\arraystretch}{1.2}
\centering
\setlength\tabcolsep{3pt}
\begin{tabular}{|c|ccccccc|} 
\hline
$-$ & $(0,0)$ & $(1,0)$ & $(2,0)$ & $(3,0)$ & $(4,0)$ & $(5,0)$ & $(6,0)$ \\ \hline
$(40,1)$ & $\textbf{(40,1)}$ & $\textbf{(39,1)}$ &	$\textbf{(38,1)}$ & $\textbf{(37,1)}$ &	$\textbf{(36,1)}$	& $\textbf{(35,1)}$ & $\textbf{(34,1)}$ \\
$(66,0)$ & $ (66,0)$	& $(65,0)$ & $(64,0)$ & $(63,0)$	& $(62,0)$	& $(61,0)$	& $(60,0)$ \\
$(18,1)$ & $ \textbf{(18,1)}$ & $	\textbf{(17,1)}$ & $	\textbf{(16,1)}	$ & $\textbf{(15,1)}	$ & $\textbf{(14,1)}$ & $	\textbf{(13,1)}$ &$	\textbf{(12,1)}$ \\
$(44,0)$ & $ (44,0)$ & $(43,0)$ & $(42,0)$ & $(41,0)	$ & $(40,0)	$ & $(39,0)	$ & $(38,0)$ \\
$(70,1)$ & $ \textbf{(70,1)}$ & $	\textbf{(69,1)}	$ & $\textbf{(68,1)}$ & $	\textbf{(67,1)}	$ & $\textbf{(66,1)}	$ & $\textbf{(65,1)}$ & $	\textbf{(64,1)}$ \\
$(22,0)$ & $ (22,0)$ & $	(21,0)$ & $(20,0)$  & $	(19,0)$ & $(18,0)$ & $(17,0)$ & $(16,0)$ \\
$(48,1)$ & $ \textbf{(48,1)}$ & $	\textbf{(47,1)}$ & $	\textbf{(46,1)}	$ & $\textbf{(45,1)}	$ & $\textbf{(44,1)}$ & $	\textbf{(43,1)}	$ & $\textbf{(42,1)}$ \\
\hline
\end{tabular}
	
\caption{$\Delta(A_1,A_0)$ for the construction in $\mathbb{Z}_{74} \times \mathbb{Z}_2$ from Theorem \ref{thm:noncyclic}}
\label{table:A_1A_0}
\end{table}

\begin{table}[h]
\renewcommand{\arraystretch}{1.2}
\centering
\setlength\tabcolsep{3pt}
\begin{tabular}{|c|ccccccc|} 
\hline
$-$ & $(40,1)$ & $(66,0)$ &$(18,1)$&$ (44,0)$&$ (70,1) $&$(22,0) $&$(48,1)$ \\ \hline
$(67,0)$ & $ \textbf{(27,1)}	$&$(1,0)	$&$\textbf{(49,1)}	$&$(23,0)	$&$\textbf{(71,1)}	$&$(45,0)	$&$\textbf{(19,1)}$ \\
$(20,1)$ & $ (54,0)	$&$\textbf{(28,1)}	$&$(2,0)	$&$\textbf{(50,1)}	$&$(24,0)	$&$\textbf{(72,1)}	$&$(46,0)$ \\
$(47,0)$ & $ \textbf{(7,1)}	$&$(55,0)	$&$\textbf{(29,1)}	$&$(3,0)	$&$\textbf{(51,1)}	$&$(25,0)	$&$\textbf{(73,1)}$ \\
$(0,1)$ & $ (34,0)	$&$\textbf{(8,1)}	$&$(56,0)	$&$\textbf{(30,1)}	$&$(4,0)	$&$\textbf{(52,1)}$&$	(26,0)$ \\
$(27,0)$ & $ \textbf{(61,1)}	$&$(35,0)	$&$\textbf{(9,1)}	$&$(57,0)	$&$\textbf{(31,1)}	$&$(5,0)	$&$\textbf{(53,1)}$ \\
$(54,1)$ & $ (14,0) $&$	\textbf{(62,1)}	$&$(36,0)	$&$\textbf{(10,1)}	$&$(58,0)	$&$\textbf{(32,1)}	$&$(6,0)$ \\
$(7,0)$ & $ \textbf{(41,1)}	$&$(15,0)	$&$\textbf{(63,1)}	$&$(37,0)	$&$\textbf{(11,1)}	$&$(59,0)	$&$\textbf{(33,1)}$ \\
\hline
\end{tabular}
	
\caption{$\Delta(A_2,A_1)$ for the construction in $\mathbb{Z}_{74} \times \mathbb{Z}_2$ from Theorem \ref{thm:noncyclic}}
\label{table:A_2A_1}
\end{table}
\begin{table}[h]
\renewcommand{\arraystretch}{1.2}
\centering
\setlength\tabcolsep{3pt}
\begin{tabular}{|c|ccccccc|} 
\hline
$-$ & $(67,0)$&$ (20,1)$&$(47,0) $&$(0,1)$&$ (27,0)$&$ (54,1)	$&$(7,0)$ \\ \hline
$(0,0)$ & $ (7,0)	$&$\textbf{(54,1)}	$&$(27,0)	$&$\textbf{(0,1)}	$&$(47,0)	$&$\textbf{(20,1)} $&$	(67,0)$ \\
$(1,0)$ & $ (8,0)	$&$\textbf{(55,1)}	$&$(28,0)	$&$\textbf{(1,1)}	$&$(48,0)	$&$\textbf{(21,1)}	$&$(68,0)$ \\
$(2,0)$ & $ (9,0)	$&$\textbf{(56,1)}	$&$(29,0)	$&$\textbf{(2,1)}	$&$(49,0)	$&$\textbf{(22,1)}$&$	(69,0)$ \\
$(3,0)$ & $ (10,0)	$&$\textbf{(57,1)}	$&$(30,0)	$&$\textbf{(3,1)}	$&$(50,0)	$&$\textbf{(23,1)}$&$	(70,0)$ \\
$(4,0)$ & $ (11,0)	$&$\textbf{(58,1)}	$&$(31,0)$&$	\textbf{(4,1)}	$&$(51,0)	$&$\textbf{(24,1)}	$&$(71,0)$ \\
$(5,0)$ & $ (12,0)	$&$\textbf{(59,1)}	$&$(32,0)	$&$\textbf{(5,1)}	$&$(52,0)	$&$\textbf{(25,1)}$&$	(72,0)$ \\
$(6,0)$ & $ (13,0)	$&$\textbf{(60,1)}	$&$(33,0) $&$	\textbf{(6,1)}	$&$(53,0)	$&$\textbf{(26,1)}	$&$(73,0)$ \\
\hline
\end{tabular}
	
\caption{$\Delta(A_0,A_2)$ for the construction in $\mathbb{Z}_{74} \times \mathbb{Z}_2$ from Theorem \ref{thm:noncyclic}}
\label{table:A_0A_2}
\end{table}
\end{example}

We now provide the proof of Theorem \ref{thm:noncyclic}.  The reader may find it helpful to consult the tables for the $l=3$ and $l=7$ cases as they navigate the proof.

\begin{proof}
Since $l \equiv 3 \mod 4$, $(3l^2+1)/2$ is even and hence $\mathbb{Z}_{\frac{3l^2+1}{2}}\times\mathbb{Z}_2$ is not a cyclic group.  We will show that $\Delta(A_1,A_0) \cup \Delta(A_2,A_1) \cup \Delta(A_0,A_2)=G \setminus \{0\}$; note that disjointness of the three sets then follows from the fact that $0$ does not occur in the multiset of differences.  Since the difference multiset contains $3l^2$ elements by construction, it will suffice to show that each non-identity group element occurs at least once.\\
For $a, b \in \mathbb{Z}_{\frac{3l^2+1}{2}}$ and $x \in \mathbb{Z}_2$, we will adapt our previous interval notation as follows: we denote the set $\{(a,x), (a+1,x), \ldots, (b,x)\}$ by $[a,b]\times \{x\}$. \\
First, we determine the difference multisets. 
\begin{align*}
\Delta  (A_1,A_0) &= \bigcup_{i=0}^{l-1}\bigcup_{j=0}^{l-1}\{z(i-1)-l-i-j,i+1)\}\\
&= \bigcup_{i=0}^{l-1}[z(i-1)-2l+1-i,z(i-1)-l-i]\times\{i+1\}
\end{align*}
We may view these differences as consisting of $l$ length-$l$ ``runs" of consecutive elements in the first coordinate, which occur horizontally in the difference table, indexed by $i$. Each run has fixed second coordinate, and the parity of the second coordinate alternates as $i$ takes values from $0$ to $l-1$.

Next we have:
\begin{align*}
\Delta  (A_2,A_1) &= \bigcup_{i=0}^{l-1}\bigcup_{j=0}^{l-1}\{(zi-l-z(j-1)+l+j,i-j+1)\}\\
&= \bigcup_{i=0}^{l-1}\bigcup_{j=0}^{l-1}\{(z(i-j+1)+j,i-j+1)\}.
\end{align*}
We let $k = i-j+1$: since $i$ and $j$ vary from $0$ to $l-1$, $k$ takes values from  $-l+2$ to $l$.  We change the indices from $i$ and $j$ to $k$ and $j$.  For $-l+2 \leq k \leq l$, let  $$J_k=\{j: 0 \leq j \leq l-1 \mbox{ and } 0 \leq k+j-1 \leq l-1\}=[0,l-1] \cap [1-k,l-k]$$
so that
$$\Delta  (A_2,A_1) = \bigcup_{k=-l+2}^{l}\bigcup_{j \in J_k} \{(z k+j,k)\}.$$
For $-l+2 \leq k \leq 0$, we have $J_k=[1-k,l-1]$, while for $1 \leq k \leq l$ we have $J_k=[0,l-k].$  Thus we have:
\[\Delta  (A_2,A_1) 
= \bigcup_{k=-l+2}^{0}\bigcup_{j=1-k}^{l-1}\{(zk+j,k)\}\cup\bigcup_{k=1}^{l}\bigcup_{j=0}^{l-k}\{(zk+j,k)\}\]
\[= \bigcup_{k=-l+2}^{0}\left([zk+1-k,zk+l-1]\times\{k\}\right)\cup\bigcup_{k=1}^{l}\left([zk,zk+l-k]\times\{k\}\right)\]
We may view these differences as ``runs" of consecutive elements (in the first coordinate) occurring diagonally in the difference table, indexed by $k$ and ``wrapping around" the table to $k-(l-1)$. The set $J_k$ allows for the adjustment of the length of the differences depending on which diagonal we are considering. Again, we have alternation between $0$ and $1$ in the second coordinate for each run.

Finally:
\begin{align*}
\Delta(A_0,A_2) &=\bigcup_{i=0}^{l-1}\bigcup_{j=0}^{l-1}\{(i-zj+l,j)\}\\
&=\bigcup_{j=0}^{l-1}[-zj+l,-z j+2l-1]\times\{j\}
\end{align*}
These differences we may view as ``runs" of consecutive elements (in the first coordinate) occurring vertically in the difference table, indexed by $j$, with alternation in the second coordinate.

Having obtained expressions for each part of the difference multiset, we now check that the difference multiset has one occurrence of each non-identity element in $\mathbb{Z}_{\frac{3l^2+1}{2}} \times \{0\}$.  We consider the differences corresponding to the following indices, where $ 0 \leq s \leq \frac{l-3}{2}$.  
\begin{itemize}
\item $k=-2s$ in $\Delta(A_2,A_1)$: the difference multiset obtained is
\begin{flalign*}
&[z(-2s)+1+2s,z(-2s)+l-1]\times\{-2s\}\\
&=[(3l-1)s+2s+1,(3l-1)s+l-1]\times\{0\}\\
&=[(3l+1)s+1,(3l-1)s+l-1]\times\{0\}.
\end{flalign*}
Here we have applied the simplification $2z=\frac{3(l^2-2l+1)}{2}=\frac{3l^2+1}{2}+\frac{2-6l}{2} \equiv 1-3l \mod \frac{3l^2+1}{2}$.
\item $j =2s$ in $\Delta(A_0,A_2)$: the difference multiset obtained is
\begin{flalign*}
&[(-z (2s)+l,-z (2s)+2l-1]\times\{2s\}\\
&=[((3l-1)s+l,(3l-1)s+2l-1]\times\{0\}.
\end{flalign*}
\item $k=l-1-2s$ in $\Delta(A_2,A_1)$: since $2 \leq l-1-2s \leq l-1$, 
the difference multiset obtained is
\begin{flalign*}
&[z(l-1-2s),z(l-1-2s)+l-(l-2s-1)]\times\{l-1-2s\}\\
&=[(3l-1)s+z(l-1), (3l-1)s+z(l-1)+2s+1]\times\{0\}\\
&=[(3l-1)s+2l,((3l-1)s+2l+2s+1]\times\{0\}\\
&=[(3l-1)s+2l,(3l+1)s+2l+1]\times\{0\}.
\end{flalign*}
Here we have applied the simplification $z(l-1)=\frac{3}{4}(l-1)^3 \equiv (\frac{l-1}{2})(1-3l)=-(\frac{3l^2+1}{2})+2l \equiv 2l \mod \frac{3l^2+1}{2}$.
\item $i=l-2-2s$ in $\Delta(A_1,A_0)$: the difference multiset obtained is
\begin{flalign*}
&[z(l-3-2s)-2l+1-(l-2-2s),\\
&z(l-3-2s)-l-(l-2-2s)]\times\{l-1-2s\}\\
&=[(3l-1)s + (5l-1)-3l+3+2s),(3l-1)s+(5l-1)-2l+2+2s]\times\{0\}\\
&=[(3l+1)s+2l+2,(3l+1)s+3l+1]\times\{0\}.
\end{flalign*}
\end{itemize}
The union of these is $[(3l+1)s+1,(3l+1)(s+1)]\times\{0\}$.
As $s$ ranges from $ 0 \text{ to } \frac{l-3}{2}$ we obtain one copy of:
\[[1,\frac{3l^2+1}{2}-l-1]\times\{0\}\]
Finally, we take $j=l-1$ in $\Delta(A_0,A_2)$.  This contributes the set of differences 
$$[\frac{3l^2+1}{2}-l,\frac{3l^2+1}{2}-1]\times\{0\}.$$
Now we consider the elements of $\mathbb{Z}_{\frac{3l^2+1}{2}}\times\{1\}$. We will use four sets of indices, of which the first and third set apply only to $l \geq 7$.  We omit the details of the calculations, since they are very similar to those above.

First, we consider the differences corresponding to the following indices, where $0 \leq s \leq \frac{l-7}{4}$ (note these terms do not occur when $l=3$):
\begin{itemize}
    \item $j=\frac{l-1}{2}+2s$ in $\Delta(A_0,A_2)$: differences $[(3l-1)s,(3l-1)s+l-1)]\times\{1\};$
    \item $k=\frac{l-1}{2}-2s$ in $\Delta(A_2,A_1)$: differences $[(3l-1)s+l,(3l+1)s+\frac{3l+1}{2}]\times\{1\};$
    \item $i=\frac{l-3}{2}-2s$ in $\Delta(A_1,A_0)$: differences $[(3l+1)s+\frac{3l+3}{2},(3l+1)s+\frac{5l+1}{2}]\times\{1\};$
    \item $k=\frac{-l-3}{2}-2s$ in $\Delta(A_2,A_1)$: differences $[(3l+1)s+\frac{5l+3}{2},(3l-1)s+3l-2]\times\{1\}.$
\end{itemize}
The union of these is $[(3l-1)s,(3l-1)(s+1)-1]\times\{1\}.$ As $s$ ranges from $0$ to $\frac{l-7}{4}$ we obtain one copy of:
\[[0,\frac{3l^2-10l-1}{4}]\times\{1\}\]
Next we take $j=l-2$ in $\Delta(A_0,A_2)$, $k=1$ in $\Delta(A_2,A_1)$, and $i= 0$ in $\Delta(A_1,A_0)$; these contribute the differences:
\[[\frac{3l^2-10l+3}{4},\frac{3l^2-6l-1}{4}]\times\{1\};\]
\[[\frac{3l^2-6l+3}{4},\frac{3l^2-2l-1}{4}]\times\{1\};\]
\[[\frac{3l^2-2l+3}{4},\frac{3l^2+2l-1}{4}]\times\{1\};\]
respectively.  (Note that, for $l=3$, this range of differences is $[0,8]\times\{1\}$.)

We now take the differences corresponding to the following indices, where $0 \leq s \leq \frac{l-7}{4}$ (as in the first case, note these terms do not occur when $l=3$):
\begin{itemize}
    \item $k=l-2s$ in $\Delta(A_2,A_1)$: differences $[(3l-1)s+\frac{3l^2+2l+3}{4},(3l+1)s+\frac{3l^2+2l+3}{4}]\times\{1\};$
    \item $i=l-1-2s$ in $\Delta(A_1,A_0)$: differences $[(3l+1)s+\frac{3l^2+2l+7}{4},(3l+1)s+\frac{3l^2+6l+3}{4}]\times\{1\};$
    \item $k=-1-2s$ in $\Delta(A_2,A_1)$: differences $[(3l+1)s+\frac{3l^2+6l+7}{4},(3l-1)s+\frac{3l^2+10l-5}{4}]\times\{1\};$
    \item $j=1+2s$ in $\Delta(A_0,A_2)$: differences $[(3l-1)s+\frac{3l^2+10l-1}{4},(3l-1)s+\frac{3l^2+14l-5}{4}]\times\{1\}.$
\end{itemize}
The union of these is $[(3l-1)s+\frac{3l^2+2l+3}{4},(3l-1)(s+1)+\frac{3l^2+2l-1}{4}]\times\{1\}$
As s ranges from $0$ to $\frac{l-7}{4}$ this gives one copy of:
\[[\frac{3l^2+2l+3}{4},\frac{3l^2-4l+1}{2}]\times\{1\}\]

Finally we take $k=\frac{l+3}{2}$ in $\Delta(A_2,A_1)$, $i = \frac{l+1}{2}$ in $\Delta(A_1,A_0)$ and $k=\frac{-l+1}{2}$ in $\Delta(A_2,A_1)$.  These contribute the differences:
\[[\frac{3l^2-4l+3}{2},\frac{3l^2-3l}{2}]\times\{1\};\]
\[[\frac{3l^2-3l+2}{2},\frac{3l^2-l}{2}]\times\{1\};\]
\[[\frac{3l^2-l+2}{2},\frac{3l^2+1}{2}-1]\times\{1\};\]
respectively.  (Note that, for $l=3$, this range of differences is $[9,13]\times\{1\}$.)

We have shown that $\Delta(A_1,A_0) \cup \Delta(A_2,A_1) \cup \Delta(A_0,A_2)$ comprises one copy of each non-identity element in $\mathbb{Z}_{\frac{3l^2+1}{2}}\times\mathbb{Z}_2$, and so $(A_0,A_1,A_2)$ form a $(3l^2+1,3,l,1)$-CEDF in $\mathbb{Z}_{\frac{3l^2+1}{2}}\times\mathbb{Z}_2.$
\end{proof}

We next present the first example of a CEDF in a non-abelian group.  This was found by computational search in GAP \cite{GAP} using a constraint-satisfaction technique.

\begin{example}
Denote by $D_{28}$ the dihedral group with the standard presentation $\langle r,s: \mathrm{ord}(r)=14, \mathrm{ord}(s)=2, srs=r^{-1} \rangle$.  Consider the following subsets of $D_{28}$:  
\begin{itemize}
\item $A_0=\{id, r^{11}, r^8\}$
\item $A_1=\{ r^4,sr^2, sr^6\}$
\item $A_2=\{r^3,r^5, sr^4\}$.
\end{itemize}
Then $(A_0,A_1,A_2)$ is a $(28,3,3,1)$-CEDF in ${D}_{28}$.
It can be verified directly that
\begin{itemize}
\item $\Delta(A_1,A_0)=\{r^4,r^{10},r^7,sr^2, sr^8, sr^5, sr^6, sr^{12}, sr^9\}$
\item $\Delta(A_2,A_1)=\{r^{13}, sr^{13}, sr^3, r, sr^{11}, sr, s, r^{12}, r^2\}$
\item $\Delta(A_0,A_2)=\{r^{11}, r^9, sr^4, r^8, r^6, sr^{7}, r^5, r^3, sr^{10}\}$.
\end{itemize}
\end{example}
Observe that this non-abelian CEDF has the same parameters as the abelian CEDF obtained from Theorem \ref{thm:noncyclic} in $\mathbb{Z}_{14} \times \mathbb{Z}_2$ when $l=3$.
\end{subsection}

\begin{subsection}{Inequivalent CEDFs in cyclic groups}

In \cite{PatSti2}, a construction is given for $(ml^2+1,m,l,1)$-CEDFs in the cyclic group $\mathbb{Z}_{ml^2+1}$ using $\alpha$-valuations, in the case when $m$ is even.  In this section, we define the notion of equivalence for CEDFs and exhibit a CEDF construction in cyclic groups which provides inequivalent CEDFs with the same parameters.

\begin{definition}
Let $G$ be an abelian group. Let $\mathcal{A}=(A_0,\ldots, A_{m-1})$ and $\mathcal{B}=(B_0,\ldots, B_{m-1})$ be two $(n,m,l,\lambda)$-$c$-CEDFs in $G$. We shall say that $\mathcal{A}$ is equivalent to $\mathcal{B}$ if there exist an automorphism $\sigma$ of $G$ and an element $\beta \in G$ such that, for all $0 \leq i \leq m-1$, $B_{i+c \mod m}=\sigma(A_i) + \beta$ for some $c \in \{0,\ldots,m-1\}$.  If $G=\mathbb{Z}_n$, this simplifies to the following definition: $\mathcal{A}$ is equivalent to $\mathcal{B}$ if there exist $\alpha, \beta \in \mathbb{Z}_n$, where $\alpha$ is a unit of the ring $\mathbb{Z}_n$, such that for all $0 \leq i \leq m-1$, $B_{i+c \mod m}=\alpha A_i + \beta$ for some $c \in \{0,\ldots,m-1\}$.
\end{definition}

For a subset $D$ of a group $G$, we define the multiset of internal differences by $\Delta(D)=\{x-y: x \neq y \in D\}$.  Recall that for subsets $A,B$ of $G$, the multiset $\Delta(A,B)$ is given by $\{x-y: x \in A, y \in B\}$. Note that $\Delta(A,A)=\Delta(A)+ |A| \{0\}$.

We will need the following prior result, which we state without proof.

\begin{lemma}\label{lem:intdiffs}
Let $n \in \mathbb{N}$ and $G=\mathbb{Z}_n$.
\begin{itemize}
\item[(i)]  For an interval $I=[0,k]$ in $G$ with $k<n/2$, the maximum multiplicity of an element in $\Delta(I,I)$ is $k+1$ (attained by element $0$).
\item[(ii)] For an interval $I=[0,k]$ in $G$ with $k<n/2$, the maximum multiplicity of an element in $\Delta(I)$ is $k$ (attained by elements $\pm 1$).
\item[(iii)] For a subset $A$ of $G$ and $\alpha \in G$, the following hold:
\begin{itemize}
 \item [1)] $\Delta(A) = \Delta(A+\alpha)$ and $\Delta(A,A) = \Delta(A+\alpha, A+ \alpha)$ 
\item [(2)] $\Delta(\alpha A) = \alpha (\Delta(A))$ and $\Delta(\alpha A, \alpha A) = \alpha (\Delta(A,A))$
\item [(3)] $\Delta(A+\alpha, A) =\Delta(A,A)+\alpha$.
\end{itemize}
\end{itemize}
\end{lemma}

\begin{theorem}\label{thm:m=4}
Let $l \in \mathbb{N}$ and let $d$ be a divisor of $l$.  Define $\mathcal{A}=(A_0,A_1,A_2,A_3)$ to be the following (ordered) collection of sets in $\mathbb{Z}_{4l^2+1}$:
\begin{itemize}
\item $A_0=\{i: 0 \leq i \leq l-1\}$;
\item $A_1=\bigcup_{k=0}^{d-1}\{\frac{l^2(2k+1)}{d}+(i+1)l: 0 \leq i \leq \frac{l}{d}-1\}$;
\item $A_2=A_0+\frac{l^2}{d}=\{\frac{l^2}{d}+i:0 \leq i \leq l-1\}$;
\item $A_3=A_1+2l^2=\bigcup_{k=0}^{d-1}\{\frac{l^2(2(k+d)+1)}{d}+(i+1)l: 0 \leq i \leq \frac{l}{d}-1\}$.
\end{itemize}
Then
\begin{itemize}
\item[(i)] $\mathcal{A}$ is a $(4l^2+1,4,l,1)$-CEDF in $\mathbb{Z}_{4l^2+1}$.
\item[(ii)] For any two distinct proper divisors $d_1, d_2$ of $l$ such that $l \neq d_1 d_2$, the two $(4l^2+1,4,l,1)$-CEDFs obtained in (i) are non-equivalent. In particular, this guarantees at least two non-equivalent $(4l^2+1,4,l,1)$-CEDFs for any composite $l$.
\end{itemize}
\end{theorem}
\begin{proof}
For (i), we order the elements of $\mathbb{Z}_{4l^2+1}$ in the usual way as $0< \cdots <4l^2$.  Since $d(l-1)<l^2$ (as $d|l$), we have that all elements of $A_0$ precede all elements of $A_2$ (which lie between $l^2/d$ and $l^2/d+l-1$), which in turn precede all elements of $A_1$ (which lie between $l^2/d+l$ and $2d(l^2/d)=2l^2$), which in turn precede all elements of $A_3$ (which lie between $(2d+1)l^2/d+l$ and $4l^2$).  All four sets are pairwise disjoint and have no internal repetition of elements.
We now show that the multiset equation $\cup_{i=0}^{3} \Delta(A_{i+1 \mod 4}, A_i)= (\mathbb{Z}_{4l^2+1} \setminus \{0\})$ holds for $\mathcal{A}$.  We will show that $\Delta(A_1,A_0) \cup \Delta(A_0,A_3)=[1,2l^2]$ and $\Delta(A_2,A_1) \cup \Delta(A_3,A_2)=[2l^2+1,4l^2]$.
\begin{align*}
\Delta(A_1,A_0) &= \bigcup_{k=0}^{d-1}\{\frac{l^2(2k+1)}{d}+(i+1)l-j\}: 0 \leq i \leq \frac{l}{d}-1, 0 \leq j \leq l-1 \}\\
&=\bigcup_{k=0}^{d-1}\bigcup_{i=0}^{\frac{l}{d}-1}\bigcup_{j=0}^{l-1}\{\frac{l^2(2k+1)}{d}+(i+1)l -j \}\\
&=\bigcup_{k=0}^{d-1}\bigcup_{i=0}^{\frac{l}{d}-1}[\frac{l^2(2k+1)}{d}+il+1, \frac{l^2(2k+1)}{d}+(i+1)l ]\\
&=\bigcup_{k=0}^{d-1} [\frac{l^2(2k+1)}{d}+1, \frac{l^2(2k+2)}{d}]
\end{align*}
Using the fact that $-2l^2 \equiv 2l^2+1 \mod 4l^2+1$, we have
\begin{align*}
\Delta(A_0,A_3) &=\Delta(A_0,A_1+2l^2)\\
&=-\Delta(A_1,A_0)-2l^2\\
&= 2l^2+1 - \Delta(A_1,A_0)\\
&= 2l^2+1 + \bigcup_{k=0}^{d-1} [-\frac{l^2(2k+2)}{d},-\frac{l^2(2k+1)}{d}-1]\\
&=\bigcup_{k=0}^{d-1} [\frac{l^2(2d-2k-2)}{d}+1, \frac{l^2(2d-2k-1)}{d}]\\
&=\bigcup_{k=0}^{d-1} [\frac{l^2(2k)}{d}+1, \frac{l^2(2k+1)}{d}]
\end{align*}
where in the final step we take $d-1-k$ instead of $k$.

Hence 
$$\Delta(A_1,A_0) \cup \Delta(A_0,A_3) =\bigcup_{k=0}^{d-1} [\frac{l^2(2k)}{d}+1, \frac{l^2(2k+2)}{d}]=[1,2l^2].$$
Next,
\begin{align*}
\Delta(A_2,A_1) &=\Delta(A_0+\frac{l^2}{d},A_1)\\
&=-\Delta(A_1+,A_0)+\frac{l^2}{d}\\
&=-\bigcup_{k=0}^{d-1} [\frac{l^2(2k+1)}{d}+1, \frac{l^2(2k+2)}{d}]+\frac{l^2}{d}\\
&=\bigcup_{k=0}^{d-1} [\frac{l^2(-2k)}{d}-1, \frac{l^2(-2k-1)}{d}]\\
\end{align*}
Now we add $4l^2+1$ to make the differences positive; note this reverses the direction of the union.
\begin{align*}
&=\bigcup_{k=0}^{d-1} [\frac{l^2(4d-2k-1)}{d}+1, \frac{l^2(4d-2k)}{d}]\\
&=\bigcup_{k=0}^{d-1} [\frac{l^2(2(k+d)+1)}{d}+1, \frac{l^2(2(k+d)+2)}{d}]
\end{align*}
where again in the final step we take $d-1-k$ instead of $k$.
\begin{align*}
\Delta(A_3,A_2) &=\Delta(A_1+2l^2,A_0+\frac{l^2}{d})\\
&=\Delta(A_1,A_0)+2l^2-\frac{l^2}{d}\\
&=\bigcup_{k=0}^{d-1} [\frac{l^2(2k+1)}{d}+1, \frac{l^2(2k+2)}{d}]+2l^2-\frac{l^2}{d}\\
&=\bigcup_{k=0}^{d-1} [\frac{l^2 (2(k+d))}{d}+1, \frac{l^2 (2(k+d)+1)}{d}]\\
\end{align*}
Thus
$$\Delta(A_2,A_1) \cup \Delta(A_3,A_2) =\bigcup_{k=0}^{d-1} [\frac{2l^2(k+d)}{d}+1, \frac{2l^2(k+d+1)}{d}]=[2l^2+1,4l^2]$$
which completes the proof of (i).

For part (ii), let $\mathcal{C}^d=(A_0^{d},A_1^{d},A_2^{d},A_3^{d})$ be the CEDF from part (i) corresponding to divisor $d$ of $l$.  Let $d_1,d_2$ be distinct proper divisors of $l$. If $\mathcal{C}^{d_1}$ and $\mathcal{C}^{d_2}$ were equivalent, the sets of $\mathcal{C}^{d_1}$ could be mapped onto those of $\mathcal{C}^{d_2}$ via a mapping which would preserve the list of multiplicities of the internal differences of each set.   Consider $A_0^d=[0,l-1]$ and $A_2^d=[0,l-1]+\frac{l^2}{d}$ in $\mathbb{Z}_{4l^2+1}$; by Lemma \ref{lem:intdiffs}, the maximum multiplicity of an element in $\Delta(A_0^d)$ or $\Delta(A_2^d)$ is $l-1$.  We will show that, for $A_1^d=\cup_{k=0}^{d-1} (\frac{(2k+1)l^2}{d} +l[1,l/d])$ and $A_3^d=2l^2+A_1^d$, the maximum multiplicity of an element in $\Delta(A_1^d)$ or $\Delta(A_3^d)$ is $\mathrm{max}(l-d,l-\frac{l}{d})$. Note that (since we consider only proper divisors $d$ of $l$) this is $l-1$ precisely if $d=1$.  For distinct proper divisors $d_1,d_2$ of $l$ such that $l\neq d_1 d_2$, the maximum multiplicity of an element in $\{\Delta(A_1^{d_1}), \Delta(A_3^{d_1})\}$ is different from that of an element in  $\{\Delta(A_1^{d_2}), \Delta(A_3^{d_2})\}$.   This implies that it is not possible to map $\mathcal{C}^{d_1}$ onto $\mathcal{C}^{d_2}$ as described, and so these CEDFs are not equivalent. For any composite $l$, we can take $d_1=1$ and $d_2$ any prime divisor of $l$ to obtain two non-equivalent examples.

We now prove the above claim. First consider $A_1$.  Define $J=l([0,\frac{l}{d}-1]+1)$ and $B_k=\{\frac{l^2(2k+1)}{d}+J\}$, so that $A_1=\cup_{k=0}^{d-1} B_k$ and $\Delta(A_1,A_1)=\cup_{0 \leq k_1,k_2 \leq d-1} \Delta(B_{k_2},B_{k_1})$. Observe that $\Delta(B_{k_2},B_{k_1})=\frac{2l^2(k_2-k_1)}{d}+\Delta(J,J)$.  By Lemma \ref{lem:intdiffs}, $\Delta(J,J)=l \Delta([0,\frac{l}{d}], [0,\frac{l}{d}])$. So
\begin{equation}\label{eqn:intdiff}
\Delta(A_1,A_1)=\bigcup_{0 \leq k_1,k_2 \leq d-1}\{\frac{2l^2(k_2-k_1)}{d}+\Delta(J,J)\}.
\end{equation}
All elements of $\Delta(J,J)$ lie within $l[-\frac{l}{d}+1,\frac{l}{d}-1]$ and so all elements of $\Delta(B_{k_2},B_{k_1})$ lie within $[(2(k_2-k_1)-1)\frac{l^2}{d}+l, (2(k_2-k_1)+1)\frac{l^2}{d}-l]$.  It is clear that these multisets do not overlap for distinct values of $k_2-k_1$; moreover since the interval corresponding to $k_2-k_1=d-1$ is $[(2d-3)\frac{l^2}{d}+l,(2d-1)\frac{l^2}{d}-l]$ while that corresponding to $k_2-k_1=-(d-1)$ is $[(2d+1)\frac{l^2}{d}+l+1,(2d+3)\frac{l^2}{d}-l+1]$ there is no ``wraparound" modulo $4l^2+1$ in the subtraction table for $\Delta(A_1,A_1)$. In $\Delta(B_{k_2},B_{k_1})$, the maximum multiplicity of a nonzero element is $\frac{l}{d}$ if $k_1 \neq k_2$, and $\frac{l}{d}-1$ if $k_1=k_2$.

In order to determine the maximum multiplicity of an element in $\Delta(A_1)$, we consider the maximum multiplicity of a nonzero element in $\Delta(A_1,A_1)$.  By Equation (\ref{eqn:intdiff}), we must consider the multiset $\{k_2-k_1: 0 \leq k_1,k_2 \leq d-1\}$, i.e. $\Delta([0,d-1],[0,d-1])$. Applying Lemma \ref{lem:intdiffs} to this multiset, element $0$ attains maximum multiplicity $d$ while element $1$ attains multiplicity $d-1$, the maximum multiplicity in $\Delta([0,d-1])$.  From the $d$ multisets $\Delta(B_{k_2,k_1})$ with $k_2-k_1=0$, i.e. $k_1=k_2$, the maximum possible multiplicity for a nonzero element is $d (\frac{l}{d}-1)=l-d$. From the $d-1$ multisets $\Delta(B_{k_2,k_1})$ with $k_2-k_1=1$, the maximum possible multiplicity for a nonzero element is $(d-1)\frac{l}{d}=l-\frac{l}{d}$ (and clearly all other multiplicities of nonzero elements arising from the $k_1 \neq k_2$ case do not exceed this). Hence the maximum occurrence of a nonzero element is $\mathrm{max}(l-d,l-\frac{l}{d})$.
\end{proof}

\begin{example}
Let $m=4$ and $l=4$; here $G=\mathbb{Z}_{65}$. For $d=1$, the sets from Theorem \ref{thm:m=4}  are $A_0=\{0,1,2,3\}, A_1=\{20,24,28,32\}, A_2=\{16,17,18,19\}$ and $A_3=\{52,56,60,64\}$.  For $d=2$, the sets are $A_0=\{0,1,2,3\}, A_1=\{12,16,28,32\}, A_2=\{8,9,10,11\}$ and $A_3=\{44,48,60,64\}$. 
\end{example}
\end{subsection}

\end{subsection}

\begin{subsection}{EDFs defined by undirected cycles}

In this section, we consider EDFs defined by undirected cycles.  By Theorem \ref{general:dir_to_undir}, any CEDF yields an EDF defined by an undirected cycle, while by Theorem \ref{general:ij=ji}, an EDF  $(A_1,\ldots, A_m)$ defined by a directed cycle which satisfies the extra condition $\Delta(A_i,A_j)=\Delta(A_j,A_i)$ may be used to obtain a CEDF.  However, not all $C_m$-defined EDFs correspond to CEDFs.

We first note the following, which is straightforward to prove.
\begin{theorem}
Let $q=ef+1$ be a prime power.  Then the sequence $(C_0^e,C_1^e, \ldots, C_{e-1}^e)$ of all cyclotomic classes of order $e$ in $\mathrm{GF}(q)$ forms a $(q,e,f,f)$-CEDF.
\end{theorem}
An example of this is given in Example \ref{ex:Z13}.

We provide a cyclotomic construction of $C_m$-defined EDFs which are not CEDFs. Recall that $\{C_0^e, \ldots, C_{e-1}^e\}$ are  the cyclotomic classes of order $e$ in $\mathrm{GF}(q)$ where $q=ef+1$.

\begin{theorem}
Let $q=2ab+1$, where $a,b>1$ are both odd. \\
Let $\mathcal{A}=(C_0^{2a},C_2^{2a},C_4^{2a},\cdots, C_{2(a-1)}^{2a})$. Then $\mathcal{A}$ is a $(q,a,b,b; C_{a})$-EDF which is not a CEDF.
\end{theorem}
\begin{proof}
We note that $q \equiv 3 \mod 4$; it is well-known that in this case $-1$ is a nonsquare in $\mathrm{GF}(q)$. $\mathcal{A}$ consists of the ``even" cyclotomic classes of order $2a$ (each of cardinality $b$), i.e. those multiplicative cosets of $C_0^{2a}$ which partition the squares $C_0^2$ of $\mathrm{GF}(q)$. The ``odd" classes $\{C_1^{2a}, C_3^{2a}, \ldots, C_{2a-1}^{2a}\}$ partition the nonsquares $C_1^2$;  since $-1 \in C_1^2$, this is precisely the set $\{-A: A \in \mathcal{A}\}$. Consider the difference multisets:  $\Delta(C_2^{2a},C_0^{2a})$ is the multiset union $\cup_{k=1}^b C_{r_k}^{2a}$ of $b$ (not necessarily distinct) cyclotomic classes $C_{r_1}^{2a}, \ldots, C_{r_b}^{2a}$, and for $1 \leq i \leq a-1$, the multiset $\Delta(C_{2i+2}^{2a},C_{2i}^{2a})=\alpha^{2i}\Delta(C_2^{2a},C_0^{2a})$ comprises the cyclotomic classes $\alpha^{2i}C_{r_1}^{2a}, \ldots, \alpha^{2i}C_{r_b}^{2a}$.  By this process, the difference multiset for $\mathcal{A}$ is 
$$ \cup_{k=1}^b (C_{r_k}^{2a} \cup \alpha^2 C_{r_k}^{2a} \cup \cdots \alpha^{2(a-1)} C_{r_k}^{2a})$$
where $(C_{r_k}^{2a} \cup \alpha^2 C_{r_k}^{2a} \cup \cdots \alpha^{2(a-1)} C_{r_k}^{2a})$ is $C_0^2$ if $r_k$ is even and $C_1^2$ if $r_k$ is odd. Hence the difference multiset $\cup_{i=0}^{a-1} \Delta(C_{2i+2}^{2a}, C_{2i}^{2a})$ corresponding to the ``clockwise" oriented cycle is the multiset union of $b$ sets, each from $\{C_0^2,C_1^2\}$. For the opposite orientation, the difference multiset is precisely the negative of this. Since $-C_0^2=C_1^2$, the multiset union of the differences from both orientations yields $b$ copies of $GF(q)^*$, hence $\mathcal{A}$ is an EDF defined by a cycle of length $a$, with the given parameters.
$\mathcal{A}$ cannot be a CEDF, since the difference multiset for each oriented cycle is the multiset union of an odd number of sets, each from $\{C_0^2,C_1^2\}$, which cannot give an equal number of squares and non-squares.
\end{proof}

\begin{example}
In $\mathrm{GF}(19)$, take $a=b=3$: here $2$ is a primitive element and $\mathcal{A}=\{C_0^6,C_2^6,C_4^6\}$ where $C_0^6=\{1,7,11\}$, $C_2^6=\{4,9,6\}$ and $C_4^6=\{16,17,5\}$.  Here $\Delta(C_2^6,C_0^6) \cup \Delta(C_4^6,C_2^6) \cup \Delta(C_0^6,C_4^6)$ is a multiset in which the nonzero squares of $\mathbb{Z}_{19}$ appear once and the nonsquares appear $2$ times, since  $\Delta(C_2^6,C_0^6)=C_1^6 \cup C_3^6 \cup C_4^6$. For $-(\Delta(C_2^6,C_0^6) \cup \Delta(C_4^6,C_2^6) \cup \Delta(C_0^6,C_4^6))$, nonzero squares occur twice and nonsquares occur once.  The union of both of these multisets yields $3$ copies of each nonzero elements of $\mathrm{GF}(19)$, so $\mathcal{A}$ is a $(19,3,3,3;C_3)$-EDF which is not a CEDF.
\end{example}

We also present the following examples (which are not instances of the above construction) found via computational search in GAP \cite{GAP}, using a constraint-satisfaction modelling language \cite{Akg} and solver \cite{GenJefMig}.

\begin{example}
\begin{itemize}
\item[(i)] In $\mathbb{Z}_{13}$, let $A_0=\{0,6\}, A_1=\{1,2\}, A_2=\{9,12\}$.  Then $(A_0,A_1,A_2)$ is a $(13,3,2,2, C_3)$-EDF but not a $(13,3,2,1)$-CEDF since $\Delta(A_1,A_0) \cup \Delta(A_2,A_1) \cup \Delta(A_0,A_2)$ is a multiset in which the non-zero elements of $\mathbb{Z}_{13}$ appear $0,1$ or $2$ times.  Moreover it can be checked that these sets do not correspond to an $(13,3,2,1;C_3^*)$-EDF for any orientation of $C_3$. 
\item[(ii)] In $\mathbb{Z}_{11}$, let $A_0=\{0,7\}, A_1=\{1,2\}, A_2=\{4,9\}, A_3=\{5,8\}$ and $A_4=\{3,10\}$.  Then $(A_0,A_1,A_2,A_3,A_4)$ is an $(11,5,2,4; C_5)$-EDF, but not a $(11,5,2,2)$-CEDF since $\Delta(A_1,A_0) \cup \Delta(A_2,A_1) \cup \Delta(A_3,A_2) \cup \Delta(A_4,A_3) \cup \Delta(A_0,A_4)$ is a multiset in which the nonzero elements of $\mathbb{Z}_{11}$ appear $1,2$ or $3$ times.  
\end{itemize}
\end{example}

\end{subsection}

\end{section}

\begin{section}{$H$-defined EDFs when $H$ is complete bipartite}

In this section, we consider EDFs defined by oriented and undirected complete bipartite graphs, and obtain a complete description.

Recall that our notation for a complete bipartite digraph $K_{a,b}$ has bipartition $A \cup B$ (where $|A|=a, |B|=b$), and for the oriented version, the standard set of directed edges is $\E(K_{a,b}^*):=\{ (i,j): i \in A, j \in B \}$.  We use a semi-colon to separate the sets corresponding to the vertices of $A$ from those corresponding to the vertices of $B$.

The definition of an $m$-set SEDF requires that, for each set in turn, the sets form a $K_{m-1,1}^*$-defined EDF (with the standard orientation) having that set at the centre of the star.  However, in practice only one SEDF with more than two sets is known \cite{JedLi, Wenetal}.

\begin{example}
Let $q=3^5=243$, $e=11$ and $f=22$.  Let $\mathcal{A}=\{C_i^{11}: 0 \leq i \leq 10 \}$ be the cyclotomic classes of order $11$ in $GF(243)$. In \cite{Wenetal} it is shown that $\mathcal{A}$ is a $(243,11,22,20)$-SEDF, and hence $(C^{11}_0,\ldots C^{11}_{i-1}, C^{11}_{i+1},\ldots, C^{11}_{10}; C^{11}_i)$ is a $(243,11,22,2; K_{10,1}^*)$-CEDF for each $0 \leq i \leq 10$.
\end{example}

For what follows, we require the following definitions from \cite{PatSti1}.

\begin{definition}\label{def:GSEDF}
Let $G$ be a group of order $n$ and let $m>1$.  A family of disjoint sets $\{A_1, \ldots, A_m\}$ in $G$, with $|A_i|=k_i$ for $1 \leq k \leq m$, is an $(n,m; k_1,\ldots,k_m; \lambda_1, \ldots, \lambda_m)$-GSEDF (generalised strong external difference family) if, for each $i$ with $1 \leq i \leq m$, the multiset equation
$ \bigcup_{\{j: j \neq i\}} \Delta(A_i,A_j)= \lambda_i(G \setminus \{0\})$ holds.
\end{definition} 

\begin{definition}\label{def:GEDF}
Let $G$ be a group of order $n$ and let $m>1$.  A family of disjoint sets $\{A_1, \ldots, A_m\}$ in $G$, with $|A_i|=k_i$ for $1 \leq k \leq m$, is an $(n,m; k_1,\ldots,k_m; \lambda)$-GEDF (generalised external difference family) if the multiset equation
$ \bigcup_{\{i,j: i \neq j\}} \Delta(A_i,A_j)= \lambda(G \setminus \{0\})$ holds.
\end{definition} 

Observe that SEDFs are examples of GSEDFs and EDFs are examples of GEDFs, in which all set-sizes are equal.

It turns out that EDFs defined by directed complete bipartite graphs may be completely described in terms of GSEDFs, and that EDFs defined by undirected complete bipartite graphs may be completely described in terms of GEDFs.

\begin{theorem}\label{thm:justGSEDF} Let $a,b,l \in \mathbb{N}$ and let $G$ be a group of order $n$.  The following statements are equivalent:
\begin{itemize}
\item[(i)] there exists an $(n,a+b,l,\lambda; K_{a,b}^*)$-EDF;
\item[(ii)] there exists an $(n, la+lb,1,\lambda; K_{la,lb}^*)$-EDF;
\item[(iii)] there exists an $(n,2;la,lb;\lambda,\lambda)$-GSEDF.
\end{itemize}
\end{theorem}
\begin{proof}
To see that (i) holds if and only if (ii) holds, observe that if
$$(\{v_1\}, \{v_1\}, \ldots, \{v_{la}\}; \{v_{la+1}\}, \ldots, \{v_{la+lb}\})$$
is an $(n, la+lb,1,\lambda; K_{la,lb}^*)$-EDF in $G$, then partitioning $\{v_1,\ldots, v_{la}\}$ arbitrarily into $a$ size-$l$ sets and partitioning $\{v_{la+1},\ldots, v_{la+lb}\}$ arbitrarily into $b$ size-$l$ sets yields an $(n,a+b,l,\lambda; K_{a,b}^*)$-EDF; the difference multiset in both cases is 
$\Delta(\{v_{la+1},\ldots,v_{la+lb}\},\{v_1,\ldots,v_{la}\})$. Conversely, for an $(n,a+b,l,\lambda; K_{a,b}^*)$ comprising the $l$-sets $(A_1,\ldots,A_{a}; A_{a+1},\ldots,A_{a+b})$, we may take the elements of $A_1 \cup \cdots \cup A_{a}$ as $la$ singleton sets and those of $A_{a+1} \cup \cdots \cup A_{a+b}$ as $lb$ singleton sets to obtain a  $(n, la+lb,1,\lambda; K_{la,lb}^*)$-EDF.
To see that (ii) holds precisely if (iii) holds, consider the $(n, la+lb,1,\lambda; K_{la,lb}^*)$-EDF above; take the singleton sets on each size of its bipartition to obtain two disjoint sets $X_1=\{v_1,\ldots,v_{la}\}$ and $X_2=\{v_{la+1},\ldots, v_{lb}\}$ of size $la$ and $lb$ respectively.  It is clear that $\Delta(X_2,X_1)=\Delta(X_1,X_2)=\lambda(G \setminus \{0\})$. This is precisely the requirement for an $(n,2;la,lb;\lambda,\lambda)$-GSEDF.  The converse easily follows.
\end{proof}

Many constructions for two-set GSEDFs (and in particular, for two-set SEDFs) are known (see for example \cite{HucPatWEDF,JedLi, Wenetal}). Construction 3.10 of \cite{HucPatWEDF} demonstrates that the two sets $\{0,1,\ldots,k_1-1\}$ and $\{k_1, 2k_1,\ldots,k_1 k_2\}$ form a $(k_1 k_2+1,2;k_1,k_2;1,1)$-GSEDF  in $\mathbb{Z}_{k_1 k_2+1}$.

\begin{corollary}\label{cor:Kab*}
There exists a $(l^2 ab+1, a+b,l,1; K_{a,b}^*)$-EDF in $\mathbb{Z}_{l^2 ab+1}$ for any $a,b,l \in \mathbb{N}$.
\end{corollary}
\begin{proof}
Combining Construction 3.10 of \cite{HucPatWEDF} with the process in the proof of Theorem \ref{thm:justGSEDF}, take any partition of $\{0,1,\ldots,la-1\}$ into $a$ size-$l$ sets $(A_1,\ldots,A_a)$ on one side of the bipartition, and any partition of $\{la,2la,\ldots, (lb)(la)\}$ into $b$ size-$l$ sets $(A_{a+1},\ldots,A_{a+b})$ on the other.  Then $(A_1,\ldots,A_a; A_{a+1},\ldots, A_{a+b})$ is a $(l^2 ab+1,a+b,l,1; K_{a,b}^*)$-defined EDF in $\mathbb{Z}_{l^2 ab+1}$.
\end{proof}

\begin{example}
Let $l=2$ and $a=b=3$.
\begin{itemize}
\item[(i)] In $\mathbb{Z}_{37}$, apply Corollary \ref{cor:Kab*} to obtain the $(37,6,2,1; K_{3,3}^*)$-EDF given by 
$$(\{0,1\}, \{2,3\}, \{4,5\}; \{6,12\}, \{18,24\},\{30,36\}).$$
\item[(ii)] In $\mathbb{Z}_{13}$, apply the well-known  construction for a $(q,2,(q-1)/2,(q-1)/4)$-SEDF in $\mathrm{GF}(q)$ ($q \equiv 1 \mod 4$) whose sets are the nonzero squares and the nonsquares of the field (\cite{HucPatSEDF}) to obtain the $(13,2,6,3)$-SEDF with sets $\{ 1,3,4,9,10,12\}$ and $ \{2,5,6,7,8,11\}$ in $\mathbb{Z}_{13}$.  Using the relationships of Theorem \ref{thm:justGSEDF} to appropriately partition these sets shows that $(\{1,3\}, \{4,9\}, \{10,12\}; \{2,5\}, \{6,7\}, \{8,11\})$ is a $(13,6,2,3; K_{3,3}^*)$-EDF.
\end{itemize}
\end{example}

We may obtain an analogous result to Theorem \ref{thm:justGSEDF} for the undirected case; we omit the (similar) proof.

\begin{theorem}\label{thm:justGEDF}
The following statements are equivalent:
\begin{itemize}
\item[(i)] there exists an $(n,a+b,l,\lambda; K_{a,b})$-EDF;
\item[(ii)] there exists an $(n, la+lb,1,\lambda; K_{la,lb})$-EDF;
\item[(iii)] there exists an $(n,2;la,lb;\lambda)$-GEDF.
\end{itemize}
\end{theorem}

We can use Corollary \ref{cor:Kab*} obtain an undirected construction in a larger cyclic group using the same sets, which is not a $K_{a,b}^*$-defined EDF.

\begin{corollary}
There exists a $(2l^2 ab+1, a+b,l,1; K_{a,b})$-EDF in $\mathbb{Z}_{l^2 ab+1}$ for any $a,b,l \in \mathbb{N}$.
\end{corollary}
\begin{proof}
Consider the $(l^2 ab+1,a+b,l,1; K_{a,b}^*)$-EDF in $\mathbb{Z}_{l^2 ab+1}$ from Corollary \ref{cor:Kab*}, given by $(A_1,\ldots,A_a; A_{a+1},\ldots, A_{a+b})$.  We claim that this also forms the desired construction in $\mathbb{Z}_{2l^2 ab+1}$, with $1< \cdots < l^2 ab$ now viewed as elements of $\mathbb{Z}_{2l^2 ab+1}$.  The difference multiset of the original $(l^2 ab+1,a+b,l,1; K_{a,b}^*)$-EDF is $[1,l^2 ab]$; since all elements of $A_{a+1},\ldots, A_{a+b}$ are larger than all elements of $A_1,\ldots,A_a$ (using the natural ordering), the differences were obtained via standard integer subtraction, i.e. without invoking modulo $l^2 ab$.  Hence these differences $[1,l^2 ab]$ may be considered as elements of $\mathbb{Z}_{2l^2 ab+1}$, while in the reverse direction the differences are $-[1,l^2ab]=[l^2 ab+1,2l^2 ab]$.
\end{proof}

\end{section}

\begin{section}{Further work}
It is clear that the definition of digraph-defined EDF introduced in this paper naturally leads to a large number of new research questions.  

To date, only those digraph-defined EDFs defined in terms of undirected complete graphs (the standard EDFs), undirected complete graphs with oriented stars (SEDFs) and oriented cycles (CEDFs) have appeared in the literature.  It would be of great interest to see constructions and non-existence results for other EDFs defined by natural families of graphs and digraphs.

It is also of interest to understand what role the nature of the group $G$ plays in the possible range of digraph-defined EDFs obtainable in $G$.

The original EDF definition required all subsets in the family to be pairwise disjoint, partly to ensure unique decoding in the original AMD application \cite{PatSti1}, and partly because for an EDF defined by a complete graph, overlap would yield the identity as a difference.  However, from a purely combinatorial viewpoint, there is no reason to require disjoint sets; indeed, classical difference families do not require this.  We propose the following definition:

\begin{definition}
Let $G$ be a group of order $n$ and let $m>1$.  Let $H$ be a labelled digraph on $m$ vertices $\{0,1,\ldots,m-1\}$.  A family of $l$-sets $\{A_0, \ldots, A_{m-1}\}$ in $G$ is an adjacent-disjoint $(n,m,l,\lambda; H)$-EDF if the multiset equation
$$ \bigcup_{(i,j) \in \E(H)} \Delta(A_j,A_i) = \lambda (G \setminus \{0\})$$ 
holds.
\end{definition}

It is clear that examples exist with non-disjoint sets; e.g. for even $m$, any star-defined $(n,m/2+1,l,\lambda; K_{m/2,1})$-EDF $(A_0,\ldots,A_{m/2-1}; A_{m/2})$ is an example of an adjacent-disjoint CEDF, namely the $(n,m,l,\lambda; C_m^*)$-EDF given by $(B_0,\ldots, B_{m-1})$ where $B_i=A_i$ for even $i$ and $B_i=A_{m/2}$ for odd $i$.  We give an example of such an adjacent-disjoint CEDF for which we believe that a standard CEDF with the same parameters does not exist in the same group.  In $\mathbb{Z}_{19}$, there is an $(19,4,3,2; C_4^*)$-EDF - which is also a $(19,3,3,2;K_{2,1})$-EDF - with set sequence $(A_0,A_1,A_2,A_3)$ where $A_0=\{1,5,16\}$, $A_1=A_3=\{2,8,11\}$ and $A_2=\{4,10,17\}$.  However, computer search using GAP does not find a standard CEDF in $\mathbb{Z}_{19}$ with the same parameters.  We ask whether there are families of adjacent-disjoint EDFs which can be shown to achieve parameters not achievable by EDFs with disjoint sets.

It would also be possible to consider a version of digraph-defined EDFs in which set-sizes are not required to be equal, analogous to the generalisation of EDFs to GEDFs and SEDFs to GSEDFs.

Finally, the concept of equivalence has been considered for EDFs, SEDFs and (in this paper) for CEDFs. We ask whether it is of interest to consider equivalence in the context of more general graph- and digraph-defined EDFs.

\end{section}

\subsection*{Acknowledgements}
We thank Maura Paterson for helpful comments.  We thank the anonymous referees for their feedback, which greatly improved the exposition of the paper.

\end{document}